\documentclass[a4paper,12pt]{article}
\usepackage[utf8]{inputenc}
\usepackage{a4,amsmath,amssymb,amscd,array,booktabs,amsthm}
\usepackage{graphics,longtable}
\usepackage[usenames]{color}
\usepackage{colortbl,alltt}
\usepackage{tikz,tikz-qtree,tikz-qtree-compat}
\usepackage{caption}
\usepackage{float}

\newtheorem{defi}{Definition}[section]

\newtheorem{lema}[defi]{Lemma}
\newtheorem{ex}[defi]{Example}

\newcommand{\N}{\mathbb N}

\newcommand\bla{\cellcolor{gray!99}\color{white}}
\newcommand\gra{\cellcolor{gray!60}\color{white}}
\newcommand\bra{\cellcolor{gray!20}}

\title{Computation of numerical semigroups\\ by means of seeds}
\author{Maria Bras-Amorós, Julio Fernández-González}
\date{\today}

\begin{document}
\maketitle
\noindent\let\thefootnote\relax\footnotetext{\hspace{-6.5truemm}\scriptsize 
The first author was supported by the Spanish government under grant TIN2016-80250-R and by the Catalan government under grant 2014 SGR 537.\\
The second author is partially supported by the Spanish government under grant MTM2015-66180-R.}

\begin{abstract}
\noindent For the elements of a numerical semigroup which are larger than the Frobenius number, 
we introduce the definition of \emph{\,seed\,} by broadening the notion of generator.
This new concept allows us to explore the semigroup tree in an alternative efficient way, 
since the seeds of each descendant can be easily obtained from
the seeds of its parent. The paper is devoted to presenting the results which
are related to this approach, leading to a new algorithm for computing and
counting the semigroups of a given genus.
\end{abstract}

\section*{Introduction}

Let $\N_0$ be the set of non-negative integers. A \emph{numerical semigroup} is a subset $\Lambda$ 
of~\,$\N_0$ which contains~\,$0$, is closed under addition and has finite complement,~$\N_0\setminus\Lambda$.
The elements in $\N_0\!\setminus\!\Lambda$ are the \emph{gaps} of $\Lambda$,
and the number $g\!=\!g(\Lambda)$ of gaps is the \emph{genus} of~$\Lambda$.
Thus, the unique increasing bijective map 
$$\begin{array}{ccc}
\N_0\! & \!\longrightarrow\! & \!\Lambda\\
i\,\,\, & \!\longmapsto\! & \!\lambda_i
\end{array}$$ 
indexing the \emph{non-gaps}, namely the elements in $\Lambda$, 
satisfies \,$\lambda_0=0$\, and \,$\lambda_i=i+g$\, 
for any large enough positive index \,$i$.
Let \,$k\!=\!k(\Lambda)$\, denote the smallest such index.
The \emph{conductor} of $\Lambda$ is 
$$c=c(\Lambda)=\lambda_k=k+g.$$
The trivial semigroup~\,$\N_0$ corresponds to \,$c\!=\!1$, that is, to \,$g\!=\!0$.
Otherwise, the largest gap of~$\Lambda$, which is known as its \emph{Frobenius number}, is \,$c-1$.
We call $\Lambda$ \emph{ordinary} whenever \,$c$\, equals the \emph{multiplicity} $\lambda_1$,
that is, whenever \,$g=c-1$.

\vspace{1truemm}

A \emph{generator} of $\Lambda$ is a non-gap $\sigma$ 
such that $\Lambda\!\setminus\!\{\sigma\}$ is still a semigroup,
which amounts to saying that $\sigma$ is not the sum of two non-zero non-gaps.
Any semigroup of genus \,$g\!\geq\!1$\, can be uniquely obtained from
a semigroup~$\Lambda$ of genus~$g-1$ by removing a generator \,$\sigma\geq c(\Lambda)$. This allows 
one to arrange all numerical semigroups in an infinite tree, rooted at the trivial semigroup, 
such that the nodes at depth $g$ are all semigroups of genus $g$.
This tree was already introduced in \cite{oversemigroups,fundamentalgaps} 
and later considered in \cite{Br:fibonacci,Br:bounds,BrBu}.
The first nodes are shown in Figure \ref{tree}, where each semigroup is represented by its non-zero
elements up to the conductor.

\newpage

The number $n_g$ of semigroups of genus $g$ was conjectured in \cite{Br:fibonacci} to satisfy 
$$n_{g+2} \,\geq\, n_{g+1}+n_{g}$$ 
for $g\geq 0$, and to behave asymptotically like the Fibonacci sequence. 
Zhai proved the second statement in \cite{Zhai}.
The \emph{Sloane's On-line Encyclopedia of Integer Sequences}~\cite{sloane} 
stores the known portion of the sequence $n_g$.
The first values have been computed by Medeiros and Kakutani, Bras-Amor\'os \cite{Br:fibonacci}, 
Delgado \cite{Delgado}, and Fromentin and Hivert \cite{FromentinHivert}. 
Several contributions \cite{Br:bounds,BrBu,Elizalde,Zhao,BGP,BR,Kaplan,ordinarization,ODorney}
have been made to theoretically analyze the sequence.

\begin{figure}[h]
\centering
\resizebox{\textwidth}{!}{
\begin{tikzpicture}[grow'=right, every node/.style = {align=left}]
\tikzset{level 1/.style={level distance=2cm}}\tikzset{level 2/.style={level distance=2cm}}\tikzset{level 3/.style={level distance=2cm}}\tikzset{level 4+/.style={level distance=3cm}}\Tree
[.{\{1\}} [.{\{2\}} [.{\{3\}} [.{\{4\}} [.{\{5\}} [.{\{6\}} ] [.{\{5,7\}} ] [.{\{5,6,8\}} ] [.{\{5,6,7,9\}} ] [.{\{5,6,7,8,10\}} ] ] [.{\{4,6\}} [.{\{4,7\}} ] [.{\{4,6,8\}} ] [.{\{4,6,7,8,10\}} ] ] [.{\{4,5,7\}} [.{\{4,5,8\}} ] ] [.{\{4,5,6,8\}} ] ] [.{\{3,5\}} [.{\{3,6\}} [.{\{3,6,8\}} ] [.{\{3,6,7,9\}} ] ] [.{\{3,5,6,8\}} ] ] [.{\{3,4,6\}} ] ] [.{\{2,4\}} [.{\{2,4,6\}} [.{\{2,4,6,8\}} [.{\{2,4,6,8,10\}} ] ] ] ] ]  ] ]
\end{tikzpicture}}
\caption{
}\label{tree}
\end{figure}

\vspace{5truemm}

In this article we introduce the notion of \emph{seeds}, as a generalization 
of the generators of a semigroup which are greater than the Frobenius number. This new concept allows 
us to explore the 
semigroup tree in an alternative efficient way, since the seeds of each descendant
can be easily obtained from the seeds of its parent. Sections \ref{seeds} and \ref{results} are devoted
to presenting the results which are related to this approach, leading in Section \ref{rake} to a
new algorithm for computing the (number of) semigroups of a given genus. 
The running times that we obtain are shorter than for the other algorithms which can be found in the literature.

\newpage

\section{Seeds of a numerical semigroup}\label{seeds}

Let us fix a numerical semigroup $\Lambda=\left\{\lambda_i\right\}_{i\geq 0}$ 
with conductor \,$c=\!\lambda_k$ and genus~\,$g$. For any index~\,$i\geq 0$, the set
$$\Lambda_i:=\Lambda\setminus\left\{\lambda_1,\dots,\lambda_i\right\}$$
is a semigroup of genus \,$g+i$. It has conductor \,$c$\, if and only if \,$i\leq k-1$,
and it is ordinary if and only if \,$i\geq k-1$.

\vspace{1truemm}

\begin{defi}\label{seed}{\rm
An element \,$\lambda_t$ with \,$t\geq k$\, is a \,\emph{seed}\, of $\Lambda$ if
\,$\lambda_t+\lambda_i$\, is a generator of~$\Lambda_i$ for some index \,$0\leq i<k$.
We then say that the seed \,$\lambda_t$ has \emph{order} $i$, and also that \,$\lambda_t$ is an
\emph{order-$i$ seed}.}
\end{defi}

\vspace{1truemm}

A seed of $\Lambda$ is, by definition, at least as large as the conductor,
and it has order zero precisely when it is a generator of $\Lambda$.
The order-zero seeds of a semigroup of genus \,$g$\, are in bijection with its descendants 
of genus \,$g+1$\, in the semigroup tree, which are obtained by removing exactly one of those seeds.
Order-one seeds have been named {\em strong generators}\, in other references, 
such as \cite{BrBu}. They are behind the key idea of the bounds in~\cite{Br:bounds}.

\vspace{1truemm}

\begin{lema} 
\label{cardinalseeds}
Any order-$i$ seed of \,$\Lambda$ is at most \,$c+\lambda_{i+1}-\lambda_i-1$. 
In particular, the number of order-$i$ seeds of \,$\Lambda$ is at most \,$\lambda_{i+1}-\lambda_i$.
\end{lema}

\begin{proof}
A generator of the semigroup $\Lambda_i$ must be smaller than the sum of the conductor, which is \,$c$\,
for \,$0\leq i<k$, and the multiplicity, which is $\lambda_{i+1}$.
\end{proof}

\vspace{1truemm}

\begin{defi}{\rm
The \,{\em table of seeds}\, of $\Lambda$ is a binary table whose rows are indexed by the possible seed orders.
For \,$0\leq i\leq k-1$, the $i$-th row has \,$\lambda_{i+1}-\lambda_i$\, entries, each of them corresponding to a
possible order-$i$ seed of $\Lambda$. They are defined as follows:
for \mbox{\,$0\leq j\leq \lambda_{i+1}-\lambda_i-1$,\,} the $j$-th entry in the $i$-th row
is \,$1$\, if \,$\lambda_{k+j}=c+j$\, is an order-$i$ seed of $\Lambda$, and~\,$0$\, otherwise.
The total number of entries in the table is \,$c$.}
\end{defi}

\vspace{1truemm}

\begin{ex}
\label{ex1}{\rm
For the semigroup
$$\Lambda \,=\, \N_0\setminus\left\{1,2,3,4,6,7\right\} \,=\, \left\{0,5,8,9,10,\dots\right\}$$
one has \,$\lambda_1=5$, \,$c=\lambda_2=8$,\, so $\Lambda$ may have only order-zero
and order-one seeds. The former are simply the generators of $\Lambda$ among the elements
$$\lambda_2=8, \quad \lambda_3=9, \quad \lambda_4=10, \quad \lambda_5=11, \quad \lambda_6=12,$$
whereas the latter may only be among the first three, according to Lemma~\ref{cardinalseeds}.
Although \,$\lambda_4=2\lambda_1$\, is clearly the only one of these five elements which
is not a generator of $\Lambda$, it is an order-one seed because \,$\lambda_4+\lambda_1\,=15$\,
is not the sum of two non-zero elements in the semigroup
$$\Lambda_1 \,=\, \Lambda\setminus\left\{\lambda_1\right\} \,=\, 
\N_0\setminus\left\{1,2,3,4,5,6,7\right\} \,=\, \left\{0,8,9,10,\dots\right\}\!.$$
Since $\Lambda_1$ is ordinary in this example, \,$\lambda_2+\lambda_1$\, and \,$\lambda_3+\lambda_1$\, must then also 
be generators of this semigroup, which means that both $\lambda_2$ and $\lambda_3$ are also order-one seeds.
Thus, the table of seeds of $\Lambda$ is as follows:
$$\begin{array}{|c|c|c|c|c|}\hline
1 & 1 & 0 & 1 & 1\\\hline
1 & 1 & 1 \\\cline{1-3}
\end{array}
$$
}
\end{ex}

\vspace{1truemm}

\begin{ex}
\label{ex2}{\rm
For the semigroup
$$\Lambda \,=\, \N_0\setminus\left\{1,2,3,4,5,6,7,9,12,13\right\} \,=\, \left\{0,8,10,11,14,15,16,\dots\right\}$$
one has \,$\lambda_1=8$, \,$\lambda_2=10$, \,$\lambda_3=11$, \,$c=\lambda_4=14$. 
According to Lemma~\ref{cardinalseeds}, $\Lambda$ may have at most
\,$8$, \,$2$, \,$1$\, and \,$3$\, order-$i$ seeds for \,$i=0, 1, 2, 3$, respectively, 
and the possible seeds for each order are given consecutively from $\lambda_4$ on.
The actual seeds may be sieved from scratch using Definition \ref{seed}, as in the previous example.
Specifically, the table of seeds of~$\Lambda$ can be checked to be as follows:
$$\begin{array}{|c|c|c|c|c|c|c|c|}\hline
1 & 1 & 0 & 1 & 0  & 0 & 0 &0  \\ \hline
0 & 1  \\\cline{1-2}
1      \\\cline{1-3}
1 & 1 & 1 \\\cline{1-3}
\end{array}
$$
For instance, $\lambda_4$ is not an order-one seed because \,$\lambda_4+\lambda_1\,=22=2\lambda_3$\,
is not a generator of the semigroup
$$\Lambda_1 \,=\, \Lambda\setminus\left\{\lambda_1\right\} \,=\, 
\left\{0,10,11,14,15,16,\dots\right\}\!,$$
whereas it is an order-two seed because \,$\lambda_4+\lambda_2\,=24$\,
is not the sum of two non-zero elements in the semigroup}
$$\Lambda_2 \,=\, \Lambda\setminus\left\{\lambda_1,\,\lambda_2\right\} \,=\, 
\left\{0,11,14,15,16,\dots\right\}\!.$$
\end{ex}

\vspace{1truemm}

\section{Behavior of seeds along the semigroup tree}\label{results}

Our aim in this section is to deduce the table of seeds of a descendant in the semigroup tree 
from the table of seeds of its parent.
We keep the notations from the previous section.

\vspace{1truemm}

Let us fix an order-zero seed $\lambda_s$ of $\Lambda$, so that \,$s\geq k$\, and
$$\tilde\Lambda:=\Lambda\setminus\left\{\lambda_s\right\}$$
is a semigroup of genus \,$g+1$. The elements in $\tilde\Lambda$ are
\begin{center}
$\tilde\lambda_i=\lambda_i$ \quad for \quad $0\leq i<s$,\,
\qquad
$\tilde\lambda_i=\lambda_{i+1}=\lambda_i+1$ \quad for \quad $i\geq s$,
\end{center}
and the conductor is 
$$c(\tilde\Lambda) \,=\, \tilde\lambda_s \,=\, \lambda_s+1 \,=\, s+g+1.$$
So there are \,$s$\, possible seed orders for this semigroup: an order-$i$ seed of~$\tilde\Lambda$, 
with \,$0\leq i<s$, is an element $\lambda_t$ with \,$t>s$\,
such that \,$\lambda_t+\lambda_i$\, is a generator of the semigroup
$$\tilde\Lambda_i:=\tilde\Lambda\setminus\{\tilde\lambda_1,\dots,\tilde\lambda_i\}
=\Lambda\setminus\{\lambda_1, \dots, \lambda_i, \lambda_s\}.$$

\vspace{1truemm}

Clearly, any order-$i$ seed $\lambda_t$ of $\Lambda$ with \,$t>s$ is also 
an order-$i$ seed of $\tilde\Lambda$. This corresponds to the first type of \emph{old-order
seeds} of $\tilde\Lambda$, meaning by this term
the order-$i$ seeds of $\tilde\Lambda$ with \,$0\leq i<k$.
The four types of such seeds are gathered in the next result.

\vspace{1truemm}

\begin{lema}
Let \,$i$ be an index with \,$0\leq i<k$. An element \,$\lambda_t$ with \,$t>s$\, is
an order-$i$ seed of \,$\tilde\Lambda$ if and only if one of the following 
pairwise excluding conditions holds:
\begin{itemize}
\item[{\rm (1)}] $\lambda_t$ is an order-$i$ seed of $\Lambda$.
\item[{\rm (2)}] $i<k-1$, \,$\lambda_t=\lambda_s+\lambda_{i+1}-\lambda_i$\, and
\,$\lambda_s$ is an order-$(i+1)$ seed of $\Lambda$.
\item[{\rm (3)}] $i=k-1=s-2$\, and \,$\lambda_t=\lambda_s+\lambda_k-\lambda_{k-1}$.
\item[{\rm (4)}] $i=k-1=s-1$\, and either \,$\lambda_t=\lambda_s+\lambda_k-\lambda_{k-1}$\,
or \,$\lambda_t=\lambda_s+\lambda_k-\lambda_{k-1}+1$.
\end{itemize}
\end{lema}

\begin{proof}
Let us first assume that $\lambda_t$ is an order-$i$ seed of $\tilde\Lambda$.
Then, Lemma \ref{cardinalseeds} and the definition of $\tilde\Lambda$ yield the inequality
$$\lambda_t \,\leq\, 
\lambda_s \,+\, \tilde\lambda_{i+1} \,-\, \lambda_i.$$
Furthermore, (1) is not satisfied
if and only if
$$\lambda_t \,=\, \lambda_s \,+\, \lambda_j \,-\, \lambda_i$$
for some index \,$j\geq i+1$. 
If this holds, then
\,$\lambda_{i+1} \,\leq\, \lambda_j \,\leq\, \tilde\lambda_{i+1}$,\,
where \,$\tilde\lambda_{i+1}=\lambda_{i+1}$\, if \,$i<s-1$, 
while \,$\tilde\lambda_{i+1}=\lambda_{i+1}+1$\, otherwise. In the first case,
one has \,$j=i+1$,\, which leads to the expression for $\lambda_t$ displayed in (2) and (3).
Since \mbox{\,$0\leq i<k\leq s$,} the second case can only occur for \,$i=s-1=k-1$, 
and then the only two possibilities for $\lambda_t$ are those in (4).

\vspace{1truemm}

Let us now examine when the element \,$\lambda_s +\lambda_{i+1} - \lambda_i$\, is an order-$i$ 
seed of~$\tilde\Lambda$.  
It suffices to remark that the sum \,$\lambda_s + \lambda_{i+1}$\, 
is a generator of $\tilde\Lambda_i$ if and only if it
is a generator of $\Lambda_{i+1}$.
Then, two cases must be distinguished, depending on whether \,$i<k-1$\, or \,$i=k-1$.
In the first case, the latter condition on $\lambda_s$ amounts to 
the last one in (2) by definition of seed.
As for the second case, \,$\lambda_s + \lambda_k$\, is a generator of $\Lambda_k$
if and only if \,$s$\, and \,$k$\, are related as in~(3) or~(4), since otherwise it
can be written as the sum of
two elements in that semigroup, namely $\lambda_{s-1}$ and~$\lambda_{k+1}$.

\vspace{1truemm}

To conclude the proof, one only needs to check the second possibility in (4).
Indeed, \,$2\lambda_s+1$\, is a generator of the ordinary semigroup $\tilde\Lambda_{s-1}=\Lambda_s$.
\end{proof}

\vspace{1truemm}

Assume now that the removed seed $\lambda_s$ is different from the conductor of $\Lambda$, that is, $s>k$.
If \,$s=k+1$, then there are exactly two \emph{new-order seeds} of~$\tilde\Lambda$, meaning by this 
term the 
order-$i$ seeds of $\tilde\Lambda$ with \,$k\leq i<s$. Otherwise, there are always three such seeds.
This follows right away from the next result.

\vspace{1truemm}

\begin{lema}
For an index \,$i$ at least \,$k$, the following hold:
\begin{itemize}
\item[$\cdot$] If \,$i<s-2$, then $\tilde\Lambda$ has no order-$i$ seeds.
\item[$\cdot$] If \,$i=s-2$, then the only order-$i$ seed of \,$\tilde\Lambda$ is \,$\lambda_s+1$.
\item[$\cdot$] If \,$i=s-1$, then the only order-$i$ seeds of \,$\tilde\Lambda$ are \,$\lambda_s+1$\, and \,$\lambda_s+2$.
\end{itemize}
\end{lema}

\begin{proof}
Let $\lambda_t$ be an order-$i$ seed of $\tilde\Lambda$, so that \,$i<s<t$. Since \,$i\geq k$, one has
$$\lambda_i \,=\, \lambda_{i+1}\,-\,1 \,=\, \lambda_{i+2}\,-\,2  \qquad\quad
\mathrm{and} \qquad\quad
\lambda_t \,=\, \lambda_{t-1}\,+\,1 \,=\, \lambda_{t-2}\,+\,2.$$
Then both \,$s-i$\, and \,$t-s$\, are at most $2$. Otherwise, \,$\lambda_t+\lambda_i$\, could be written 
as the sum of two elements in the semigroup~$\tilde\Lambda_i$ above:
either \,$\lambda_{t-1}+\lambda_{i+1}$\, or~\,$\lambda_{t-2}+\lambda_{i+2}$.
Furthermore, \,$t=s+1$\, must hold whenever \,$s=i+2$.

\vspace{1truemm}

For \,$i=s-2$\, and \,$i=s-1$, one has, respectively,
$$\tilde\Lambda_i=\N_0\setminus\{1,\dots,\lambda_s-2,\lambda_s\}
\qquad  \ \ \mathrm{and}
\qquad  \ \ 
\tilde\Lambda_i=\N_0\setminus\{1,\dots,\lambda_s\}.$$
The proof concludes by checking that \,$\lambda_{s+1}+\lambda_{i}=2\lambda_s-1$\, is a generator
of $\tilde\Lambda_i$ in the first case and that \,$\lambda_{s+1}+\lambda_{i}=2\lambda_s$\, and 
\,$\lambda_{s+2}+\lambda_{i}=2\lambda_s+1$\, are both generators of $\tilde\Lambda_i$ in the second case.
\end{proof}

\vspace{2truemm}

Now the table of seeds of $\tilde\Lambda$ can be easily built from the table of seeds of $\Lambda$.
According to the results in this section, this may be done as follows:
\begin{itemize}

\item[$\cdot$] In each of the $k$ rows of the table, 
$s-k+1$ new entries must be added to the right, 
and the same number of entries must be cropped starting from the left. 
This corresponds to recycling old-order seeds of type (1).

\item[$\cdot$] Then, the last right-hand new entry in the last row must be set to~$1$
exactly when \,$s\!=\!k$\, or \,$s\!=\!k+1$. Moreover, in the first case an additional right-hand entry 
with value $1$ is required.
The last right-hand new entry in any other row must be given the value, if any, of the last right-hand cropped 
entry in the row just below. All entries which remain empty are then set to~$0$.
This codifies old-order seeds of types (2), (3) and (4),
and completes all modifications to be done in the first $k$ rows.

\item[$\cdot$] Finally, to take account of new-order seeds, \,$s-k$\, rows must be added to the table if \,$s>k$.
The last of them consists of two entries with value~$1$. If~\,$s>k+1$, the last but one row consists of 
a single entry with value $1$, while each of the remaining new rows consists of a single entry with value~$0$. 

\end{itemize}

\vspace{1truemm}

\begin{ex}
\label{ex1bis}{\rm
We can compute the tables of seeds corresponding to the four descendants of the semigroup $\Lambda$ 
in {\rm Example~\ref{ex1}}, that is, the table of seeds of \,$\tilde\Lambda$\, for~\,$s=2, 3, 5, 6$,  
by graphically implementing in each case the three steps above:}

\noindent
{\footnotesize
{\renewcommand{\arraystretch}{.85}
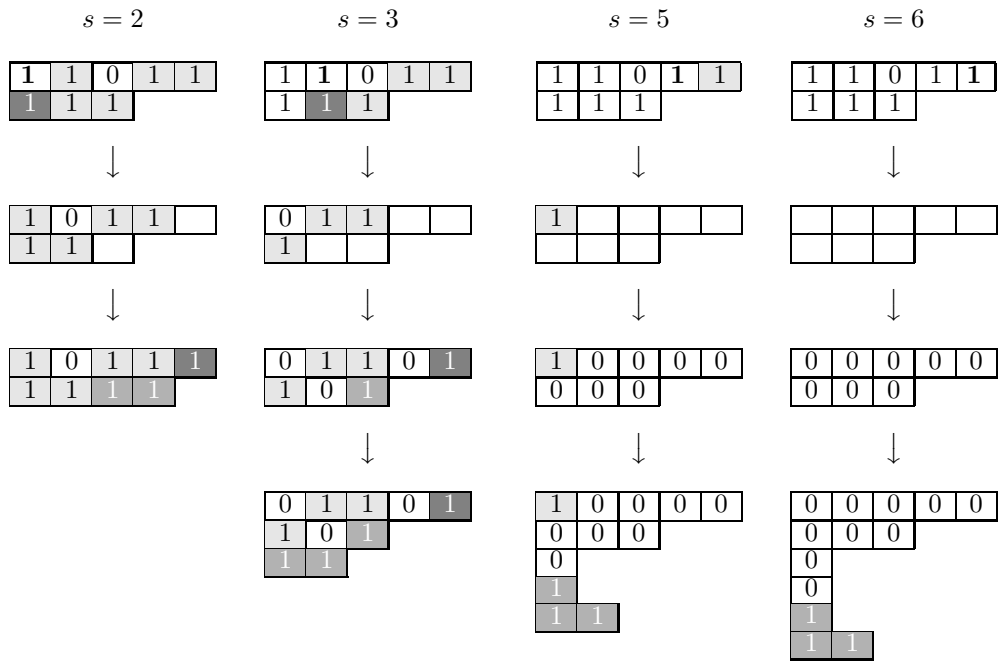
\begin{longtable}{cccc}
$s=2$ & $s=3$ & $s=5$ & $s=6$
\\\\
$\begin{array}{|c|c|c|c|c|} \hline\\[-10pt]
\!\bf 1\! & \bra 1 & 0 & \bra 1 & \bra 1 \\ \hline\\[-10pt]
\bla 1 & \bra 1 & \bra 1 \\\cline{1-3}
\end{array}$
&
$\begin{array}{|c|c|c|c|c|} \hline\\[-10pt]
1 & \!\bf 1\! & 0 & \bra 1 & \bra 1 \\ \hline\\[-10pt]
1 & \bla 1 & \bra 1 \\\cline{1-3}
\end{array}$
&
$\begin{array}{|c|c|c|c|c|} \hline\\[-10pt]
1 & 1 & 0 & \!\bf 1\! & \bra 1 \\ \hline\\[-10pt]
1 & 1 & 1 \\\cline{1-3}
\end{array}$
&
$\begin{array}{|c|c|c|c|c|} \hline\\[-10pt]
1 & 1 & 0 & 1 & \!\bf 1\!\\ \hline\\[-10pt]
1 & 1 & 1 \\\cline{1-3}
\end{array}$

\\\\

$\big\downarrow$ & $\big\downarrow$ & $\big\downarrow$ & $\big\downarrow$

\\\\

$\begin{array}{|c|c|c|c|c|} \hline\\[-10pt]
\bra 1 & 0 & \bra 1 & \bra 1 & \phantom 0\\ \hline\\[-10pt]
\bra 1 & \bra 1 & \phantom 0\\ \cline{1-3}
\end{array}$
&
$\begin{array}{|c|c|c|c|c|} \hline\\[-10pt]
0 & \bra 1 & \bra 1 & \phantom 0 & \phantom 0\\ \hline\\[-10pt]
\bra 1 & \phantom 0 & \phantom 0\\ \cline{1-3}
\end{array}$
&
$\begin{array}{|c|c|c|c|c|} \hline\\[-10pt]
\bra 1 & \phantom 0 & \phantom 0 & \phantom 0 & \phantom 0 \\ \hline\\[-10pt]
\phantom 0 & \phantom 0 & \phantom 0 \\ \cline{1-3}
\end{array}$
&
$\begin{array}{|c|c|c|c|c|} \hline\\[-10pt]
\phantom 0 & \phantom 0 & \phantom 0 & \phantom 0 & \phantom 0 \\ \hline\\[-10pt]
\phantom 0 & \phantom 0 & \phantom 0 \\ \cline{1-3}
\end{array}$

\\\\

$\big\downarrow$ & $\big\downarrow$ & $\big\downarrow$ & $\big\downarrow$

\\\\

$\begin{array}{|c|c|c|c|c|} \hline\\[-10pt]
\bra 1 & 0 & \bra 1 & \bra 1 & \bla 1\\ \hline\\[-10pt]
\bra 1 & \bra 1 & \gra 1 & \gra 1 \\ \cline{1-4}
\end{array}$
&
$\begin{array}{|c|c|c|c|c|} \hline\\[-10pt]
0 & \bra 1 & \bra 1 & 0 & \bla 1\\ \hline\\[-10pt]
\bra 1 & 0 & \gra 1\\ \cline{1-3}
\end{array}$
&
$\begin{array}{|c|c|c|c|c|} \hline\\[-10pt]
\bra 1 & 0 & 0 & 0 & 0 \\ \hline\\[-10pt]
0 & 0 & 0 \\ \cline{1-3}
\end{array}$
&
$\begin{array}{|c|c|c|c|c|} \hline\\[-10pt]
0 & 0 & 0 & 0 & 0 \\ \hline\\[-10pt]
0 & 0 & 0  \\ \cline{1-3}
\end{array}$

\\\\

& $\big\downarrow$ & $\big\downarrow$  & $\big\downarrow$

\\\\

&
\begin{tabular}{l}
$\begin{array}{|c|c|c|c|c|} \hline\\[-10pt]
0 & \bra 1 & \bra 1 & 0 & \bla 1\\ \hline\\[-10pt]
\bra 1 & 0 & \gra 1\\ \cline{1-3} \\[-10pt]
\gra 1 & \gra 1 \\\cline{1-2}
\end{array}$ \\
 \\[21pt]
\end{tabular}
&
\begin{tabular}{l}
$\begin{array}{|c|c|c|c|c|} \hline\\[-10pt]
\bra 1 & 0 & 0 & 0 & 0 \\ \hline\\[-10pt]
0 & 0 & 0 \\ \cline{1-3} \\[-10pt]
0 \\\cline{1-1} \\[-10pt]
\gra 1 \\\cline{1-2} \\[-10pt]
\gra 1 & \gra 1 \\\cline{1-2}
\end{array}$\\
 \\
\end{tabular}
&
$\begin{array}{|c|c|c|c|c|} \hline\\[-10pt]
0 & 0 & 0 & 0 & 0 \\ \hline\\[-10pt]
0 & 0 & 0  \\ \cline{1-3} \\[-10pt]
0\\ \cline{1-1}\\[-10pt]
0\\ \cline{1-1}\\[-10pt]
\gra 1\\\cline{1-2}\\[-10pt]
\gra 1 & \gra 1 \\\cline{1-2}
\end{array}$
\end{longtable}}
}
\end{ex}

\vspace{1truemm}

\begin{ex}\label{ex2bis}{\rm
We can compute the table of seeds for each of the descendants of the semigroup $\Lambda$ 
in {\rm Example~\ref{ex2}}, that is, the table of seeds of \,$\tilde\Lambda$\, for~\,$s=4, 5, 7$.}

\noindent
{\footnotesize
{\renewcommand{\arraystretch}{.85}
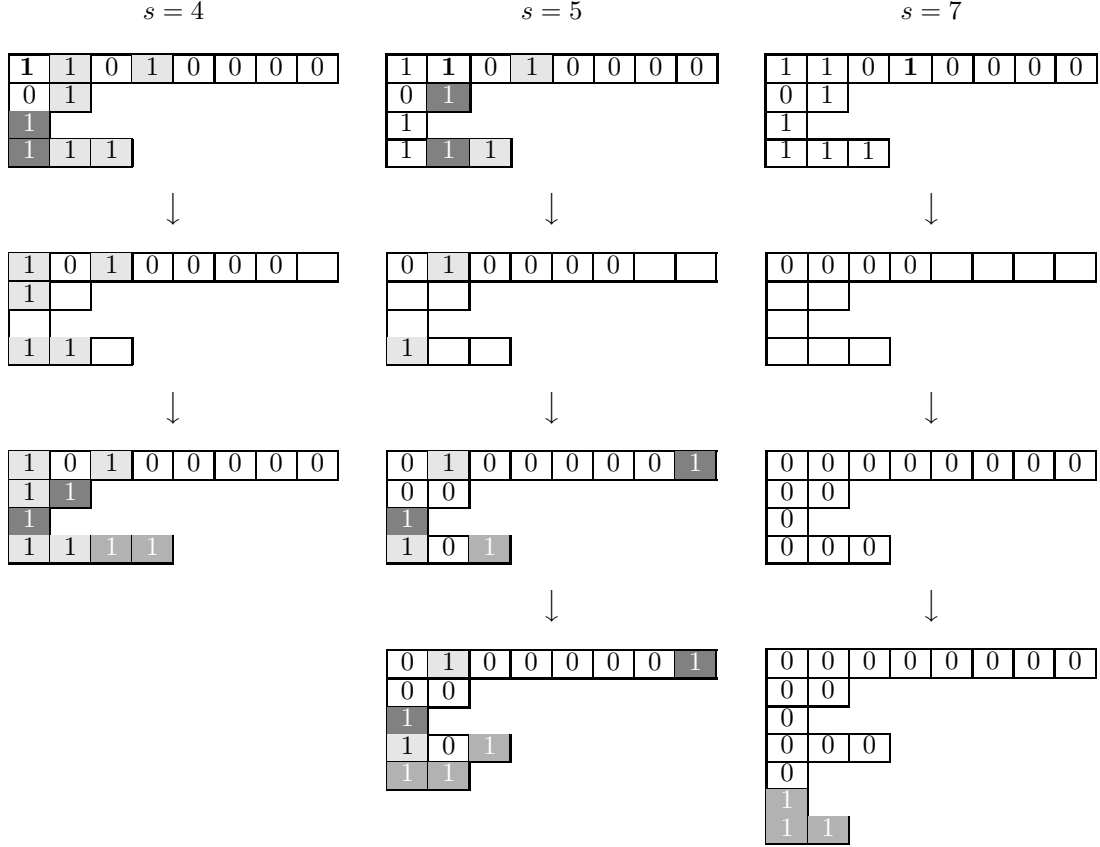
\begin{longtable}{ccc}
$s=4$ & $s=5$ & $s=7$
\\\\
$\begin{array}{|c|c|c|c|c|c|c|c|} \hline\\[-10pt]
\!\bf 1\! & \bra 1 & 0 & \bra 1 & 0  & 0 & 0 & 0  \\ \hline\\[-10pt]
0 & \bra 1  \\\cline{1-2}\\[-10pt]
\bla 1  \\ \cline{1-3}\\[-10pt]
\bla 1 & \bra 1 & \bra 1\\
\cline{1-3}
\end{array}$
&
$\begin{array}{|c|c|c|c|c|c|c|c|}\hline\\[-10pt]
1 & \bf 1 & 0 & \bra 1 & 0  & 0 & 0 & 0  \\ \hline\\[-10pt]
0 & \bla 1  \\ 
\cline{1-2}\\[-10pt]
1 \\ 
\cline{1-3}\\[-10pt]
1  & \bla 1 & \bra 1 \\\cline{1-3}
\end{array}$
&
$\begin{array}{|c|c|c|c|c|c|c|c|}\hline\\[-10pt]
1 & 1 & 0& \bf 1 & 0  & 0 & 0 &0  \\ \hline\\[-10pt]
0 & 1  \\\cline{1-2}\\[-10pt]
1      \\\cline{1-3}\\[-10pt]
1 & 1 & 1 \\\cline{1-3}
\end{array}$

\\\\

$\big\downarrow$ & $\big\downarrow$ & $\big\downarrow$

\\\\

$\begin{array}{|c|c|c|c|c|c|c|c|}\hline\\[-10pt]
\bra 1 & 0 & \bra 1 & 0  & 0 & 0 & 0 & \phantom 0 \\ \hline\\[-10pt]
\bra 1 & \phantom 0 \\\cline{1-2}\\[-10pt]
\phantom 0     \\\cline{1-3}\\[-10pt]
\bra 1 & \bra 1 & \phantom 0 \\\cline{1-3}
\end{array}$
&
$\begin{array}{|c|c|c|c|c|c|c|c|}\hline\\[-10pt]
0 & \bra 1 & 0 & 0  & 0 &0 & \phantom 0 & \phantom 0  \\ \hline\\[-10pt]
\phantom 0 & \phantom 0 \\\cline{1-2}\\[-10pt]
\phantom 0   \\\cline{1-3}\\[-10pt]
\bra 1 & \phantom 0 & \phantom 0 \\\cline{1-3}
\end{array}$
&
$\begin{array}{|c|c|c|c|c|c|c|c|}\hline\\[-10pt]
0 & 0  & 0 & 0 &  \phantom 0 & \phantom 0 & \phantom 0 & \phantom 0 \\ \hline\\[-10pt]
\phantom 0 & \phantom 0 \\\cline{1-2}\\[-10pt]
\phantom 0  \\\cline{1-3}\\[-10pt]
\phantom 0 & \phantom 0 & \phantom 0 \\\cline{1-3}
\end{array}$

\\\\

$\big\downarrow$ & $\big\downarrow$ & $\big\downarrow$

\\\\

$\begin{array}{|c|c|c|c|c|c|c|c|}\hline\\[-10pt]
\bra 1 & 0 & \bra 1 & 0  & 0 & 0 & 0 & 0 \\ \hline\\[-10pt]
\bra 1 & \bla 1 \\\cline{1-2}\\[-10pt]
\bla 1     \\\cline{1-4}\\[-10pt]
\bra 1 & \bra 1 & \gra 1 & \gra 1 \\\cline{1-4}
\end{array}$
&
$\begin{array}{|c|c|c|c|c|c|c|c|}\hline\\[-10pt]
0 & \bra 1 & 0 & 0  & 0 &0 & 0 &  \bla 1  \\ \hline\\[-10pt]
0 & 0 \\\cline{1-2}\\[-10pt]
\bla 1   \\\cline{1-3}\\[-10pt]
\bra 1 & 0 & \gra 1 \\\cline{1-3}
\end{array}$
&
$\begin{array}{|c|c|c|c|c|c|c|c|}\hline\\[-10pt]
0 & 0  & 0 & 0 & 0 & 0 & 0 & 0 \\ \hline\\[-10pt]
0 & 0 \\\cline{1-2}\\[-10pt]
0  \\\cline{1-3}\\[-10pt]
0 & 0 & 0 \\\cline{1-3}
\end{array}$

\\\\

& $\big\downarrow$ & $\big\downarrow$

\\\\

&
\begin{tabular}{l}
$\begin{array}{|c|c|c|c|c|c|c|c|}\hline\\[-10pt]
0 & \bra 1 & 0 & 0  & 0 & 0 & 0 & \bla 1  \\ \hline\\[-10pt]
0 & 0 \\\cline{1-2}\\[-10pt]
\bla 1   \\\cline{1-3}\\[-10pt]
\bra 1 & 0 & \gra 1 \\\cline{1-3}\\[-10pt]
\gra 1 & \gra 1 \\\cline{1-2}
\end{array}$\\
 \\[10pt]
\end{tabular}
&
$\begin{array}{|c|c|c|c|c|c|c|c|}\hline\\[-10pt]
0 & 0  & 0 & 0 & 0 & 0 & 0 & 0 \\ \hline\\[-10pt]
0 & 0 \\\cline{1-2}\\[-10pt]
0  \\\cline{1-3}\\[-10pt]
0 & 0 & 0 \\\cline{1-3}\\[-10pt]
0\\\cline{1-1}\\[-10pt]
\gra 1\\\cline{1-2}\\[-10pt]
\gra 1 & \gra 1 \\\cline{1-2}
\end{array}$
\end{longtable}}
}
\end{ex}

\section{The seeds descending algorithm}\label{rake}
Let us associate with a numerical semigroup $\Lambda$ two binary strings 
$$G(\Lambda)\,=\,G_0\,G_1\,\cdots\,G_{\ell}\,\cdots \qquad\quad \mathrm{and} 
\qquad\quad  S(\Lambda)\,=\,S_0\,S_1\,\cdots\,S_{\ell}\,\cdots$$
encoding, respectively, the gaps and the seeds of $\Lambda$.
Specifically, with the notations and definitions in Section \ref{seeds},
$$G_{\ell}:=\left\{\begin{array}{l}
1 \quad \mathrm{if} \ \,\ell+1\, \ \mathrm{is \ a \ gap \ of \ } \Lambda,\\[5pt]
0 \quad \mathrm{otherwise},
\end{array}
\right.$$
whereas $S(\Lambda)$ amounts to the concatenation of the rows in the table of seeds of~$\Lambda$,
that is,
$$S_{\lambda_i+j}:=\left\{\begin{array}{l}
1 \quad \mathrm{if} \ \,c+j\, \ \mathrm{is \ an \ order\mbox{-}}i \mathrm{ \ seed \ of \ } \Lambda,\\[5pt]
0 \quad \mathrm{otherwise},
\end{array}
\right.$$
for \,$i\geq 0$\, and \,$0\leq j<\lambda_{i+1}-\lambda_i$. In particular, \,$G_\ell=S_\ell=0$\, for \,$\ell\geq c$,\,
so below we identify each string with its first \,$c$\, bits:
$$G(\Lambda)\,=\,G_0\,G_1\,\cdots\,G_{c-1} \qquad\qquad \mathrm{and} 
\qquad\qquad  
S(\Lambda)\,=\,S_0\,S_1\,\cdots\,S_{c-1}.$$
The bit $G_{c-1}$ is always $0$, while \,$S_{c-3}=S_{c-2}=S_{c-1}=1$\, whenever~\,$c\geq 3$.
The pairs of binary strings \,$G(\Lambda)$, $S(\Lambda)$,\, for $\Lambda$ running over the first nodes 
in the semigroup tree, are displayed in Figure \ref{treeGS}.

\begin{figure}[H]
\centering
\resizebox{\textwidth}{11cm}{
{\bfseries\small\setlength{\tabcolsep}{1pt}
\begin{tikzpicture}[grow'=right, every node/.style = {align=left}]\tikzset{level 1/.style={level distance=1.5cm}}\tikzset{level 2/.style={level distance=2.5cm}}\tikzset{level 3/.style={level distance=3cm}}\tikzset{level 4/.style={level distance=3.5cm}}\tikzset{level 5+/.style={level distance=4cm}}\Tree
[.{\begin{tabular}{c}0\\1\\\end{tabular}} [.{\begin{tabular}{cc}1&0\\1&1\\\end{tabular}} [.{\begin{tabular}{ccc}1&1&0\\1&1&1\\\end{tabular}} [.{\begin{tabular}{cccc}1&1&1&0\\1&1&1&1\\\end{tabular}} [.{\begin{tabular}{ccccc}1&1&1&1&0\\1&1&1&1&1\\\end{tabular}} [.{\begin{tabular}{cccccc}1&1&1&1&1&0\\1&1&1&1&1&1\\\end{tabular}}  ] [.{\begin{tabular}{ccccccc}1&1&1&1&0&1&0\\1&1&1&0&1&1&1\\\end{tabular}}  ] [.{\begin{tabular}{cccccccc}1&1&1&1&0&0&1&0\\1&1&0&0&0&1&1&1\\\end{tabular}}  ] [.{\begin{tabular}{ccccccccc}1&1&1&1&0&0&0&1&0\\1&0&0&0&0&0&1&1&1\\\end{tabular}}  ] [.{\begin{tabular}{cccccccccc}1&1&1&1&0&0&0&0&1&0\\0&0&0&0&0&0&0&1&1&1\\\end{tabular}}  ]  ] [.{\begin{tabular}{cccccc}1&1&1&0&1&0\\1&1&0&1&1&1\\\end{tabular}} [.{\begin{tabular}{ccccccc}1&1&1&0&1&1&0\\1&0&1&1&1&1&1\\\end{tabular}}  ] [.{\begin{tabular}{cccccccc}1&1&1&0&1&0&1&0\\0&1&0&1&0&1&1&1\\\end{tabular}}  ] [.{\begin{tabular}{cccccccccc}1&1&1&0&1&0&0&0&1&0\\0&0&0&0&0&0&0&1&1&1\\\end{tabular}}  ]  ] [.{\begin{tabular}{ccccccc}1&1&1&0&0&1&0\\1&0&0&0&1&1&1\\\end{tabular}} [.{\begin{tabular}{cccccccc}1&1&1&0&0&1&1&0\\0&0&0&1&1&1&1&1\\\end{tabular}}  ]  ] [.{\begin{tabular}{cccccccc}1&1&1&0&0&0&1&0\\0&0&0&0&0&1&1&1\\\end{tabular}}  ]  ] [.{\begin{tabular}{ccccc}1&1&0&1&0\\1&0&1&1&1\\\end{tabular}} [.{\begin{tabular}{cccccc}1&1&0&1&1&0\\0&1&1&1&1&1\\\end{tabular}} [.{\begin{tabular}{cccccccc}1&1&0&1&1&0&1&0\\1&0&1&1&0&1&1&1\\\end{tabular}}  ] [.{\begin{tabular}{ccccccccc}1&1&0&1&1&0&0&1&0\\0&0&1&0&0&0&1&1&1\\\end{tabular}}  ]  ] [.{\begin{tabular}{cccccccc}1&1&0&1&0&0&1&0\\0&0&0&0&0&1&1&1\\\end{tabular}}  ]  ] [.{\begin{tabular}{cccccc}1&1&0&0&1&0\\0&0&0&1&1&1\\\end{tabular}}  ]  ] [.{\begin{tabular}{cccc}1&0&1&0\\0&1&1&1\\\end{tabular}} [.{\begin{tabular}{cccccc}1&0&1&0&1&0\\0&1&0&1&1&1\\\end{tabular}} [.{\begin{tabular}{cccccccc}1&0&1&0&1&0&1&0\\0&1&0&1&0&1&1&1\\\end{tabular}} [.{\begin{tabular}{cccccccccc}1&0&1&0&1&0&1&0&1&0\\0&1&0&1&0&1&0&1&1&1\\\end{tabular}}  ]  ]  ]  ]  ]  ] ]
\end{tikzpicture}}
}
\caption{
}\label{treeGS}
\end{figure}

When transcribing and implementing the pseudocodes in this section, binary strings such as $G(\Lambda)$ and $S(\Lambda)$
are regarded, stored and manipulated as integers which are encoded in binary form. 
We use below the bitwise operations on binary strings \,$\&$\, (\emph{and}) \,and
\,$\mid$\, (inclusive \emph{or}),
as well as the \,\emph{left shift}\, \,$\ll$\,
and the \,\emph{right shift}\, \,$\gg$\, of a binary string by a non-negative integer $x$.
In terms of binary encoding of integers, the latter shift is equivalent to multiplying by $2^x$.

\vspace{1truemm}

Assume from now on that $\Lambda$ is not the trivial semigroup
and let $\tilde\Lambda=\Lambda\!\setminus\!\left\{\lambda_s\right\}$ be a descendant of $\Lambda$ 
in the semigroup tree, as in Section~\ref{results}.
Then set \,$\tilde s = s - k$,\, so that \,$0\leq \tilde s<\lambda_1$\, and \,$\lambda_s=c+\tilde s$.
The string $G(\tilde\Lambda)$ is obtained from $G(\Lambda)$ by simply 
switching the bit $G_{c+\tilde s-1}$ from $0$ to $1$.
As for $S(\tilde\Lambda)$, it can be computed from~$S(\Lambda)$ by adapting the construction
of the table of seeds of $\tilde\Lambda$ as is described in Section \ref{results},
which leads to the following rephrasing in terms of strings.

\vspace{2truemm}

\begin{lema}
Let \ $\tilde S\,=\,\tilde S_0\,\tilde S_1\,\cdots\,\tilde S_{\ell}\,\cdots$ \ be the binary string defined by
$$\tilde S_{\ell}:=\left\{\begin{array}{l}
0 \quad \mathrm{if} \quad \ell=\lambda_i+j \quad \mathrm{with} \quad 1\leq i<k, \
\,0\leq j<\operatorname{min}(\tilde s, \,\lambda_{i+1}-\lambda_i),\\[5pt]
S_{\ell} \quad \mathrm{otherwise}.
\end{array}
\right.$$
Then,\\[-20pt]
$$S(\tilde\Lambda) \,=\, (\tilde S\ll \tilde s+1) \mid \left(1\,1\,1\gg c+\tilde s-2\right)
\,=\, 
\tilde S_{\tilde s+1}\,\tilde S_{\tilde s+2}\,\cdots\,\tilde S_{c-1}
\,\overbrace{0\,\cdots\,0\,\,1}^{2\tilde s}\,1\,1.$$\\[-10pt] 
In particular, \,$S(\tilde\Lambda) \,=\, \,S_1\,\cdots\,S_{c-1}\,1\,1$\, whenever \,$\tilde s=0$.
\end{lema}

\vspace{5truemm}

The string $\tilde S$ in this result satisfies 
\,$\tilde S_{\lambda_i+j}=0$\, \,for\, \,$1\leq i<k, \ 0\leq j<\tilde s$,\,
so it can be obtained by \emph{raking} $S(\Lambda)$
to get rid of the old-order seeds which must not be recycled.
This \emph{raking}\, process is performed by means of successive one-position shifts of \,$G(\Lambda)$ to the right,
as shown in the following pseudocode:

\vspace{5truemm}

{\footnotesize
\noindent{\bf Input:} \ $c:=c(\Lambda)$, \ \,$G:=G(\Lambda)$, \ \,$S:=S(\Lambda)$, \ \,$\tilde s$\\[5pt]
\noindent{\bf Output:} \ $c(\tilde\Lambda)$, \ \,$G(\tilde\Lambda)$, \ \,$S(\tilde \Lambda)$\\[-10pt]
\begin{enumerate}\itemsep0.9truemm
\item \ $\tilde{S} := S$
\item \ $\mbox{\tt rake} := G$
\item \ {\bf from} \ $1$ \ {\bf to} \ $\tilde s$ \ {\bf do}
\item \ \qquad $\mbox{\tt rake} := \mbox{\tt rake}\gg 1$
\item \ \qquad $\tilde S := \tilde S \ \& \ \mbox{\tt rake}$
\item \ {\bf return} \ $\tilde c:=c+\tilde s+1$, \ \,$G\mid \left(1\gg\tilde c-2\right)$, 
\ \,$(\tilde S\ll \tilde s+1) \mid \left(1\,1\,1\gg\tilde c-3\right)$ 
\end{enumerate}
}

\vspace{5truemm}

\noindent This is the basic \emph{descending step}\, for an algorithm that, given a level $\gamma>g(\Lambda)$,
computes the number $n_\gamma(\Lambda)$ of descendants of $\Lambda$ having genus $\gamma$. We reproduce
two different versions for the algorithm: the first one is based on a \emph{depth first search} exploration of
the nodes in the tree, and the second one is recursive.

\vspace{5truemm}

{\footnotesize
\noindent{\bf Input:} \ $\gamma$, \ \,$G(\Lambda)$, \ \,$S(\Lambda)$, \ \,$c(\Lambda)$, 
\ \,$g(\Lambda)$, \ \,$\lambda_1$\\[5pt]
\noindent{\bf Output:} \ $n_\gamma(\Lambda)$\\[-10pt]
\begin{enumerate}\itemsep0.8truemm
\item \ $g:=g(\Lambda)$  \qquad\quad $n:=0$
\item \ $G[g]:=G(\Lambda)$ \quad \ \,$S[g]:=S(\Lambda)$ \qquad $\mbox{\tt rake}[g] := G(\Lambda)$
\item \ $c[g]:=c(\Lambda)$ \qquad $m[g]:=\lambda_1$ \qquad\quad $\tilde s_{last} := \tilde s[g] := 0$
\item \ {\bf while} \ \,$g\neq g(\Lambda)-1$\, \ {\bf do}
\item \qquad {\bf while} \ \,$\tilde s[g]<m[g]$\, \ {\bf and} \ 
\,$S[g]\,\,\&\,\,(1\gg\tilde s[g]) \,=\, 0$\, \ {\bf do} \ \,$\tilde s[g] := \tilde s[g]+1$
\item \qquad {\bf if} \ \,$\tilde s[g] = m[g]$\, \ {\bf then}
\item \qquad\qquad\qquad $g := g-1$
\item \qquad\qquad\qquad $\tilde s_{last} := \tilde s[g]$
\item \qquad\qquad\qquad $\tilde s[g] := \tilde s[g]+1$
\item \qquad {\bf else} \ {\bf if} \ \,$g=\gamma-1$\, \ {\bf then}
\item \qquad\qquad\qquad $n := n+1$
\item \qquad\qquad\qquad $\tilde s[g] := \tilde s[g]+1$
\item \qquad\qquad \ \ {\bf else}
\item \qquad\qquad\qquad $c[g+1] := c[g]+\tilde s[g]+1$
\item \qquad\qquad\qquad {\bf if} \ \,$\tilde s[g] = 0$\, \ {\bf and} \ \,$m[g]=c[g]$\, \ {\bf then} \ \,$m[g+1] := c[g+1]$
\item \qquad\qquad\qquad {\bf else} \ \,$m[g+1] := m[g]$
\item \qquad\qquad\qquad {\bf from} \ \,$\tilde s_{last}+1$\, \ {\bf to} \ \,$\tilde s[g]$\, \ {\bf do}
\item \qquad\qquad\qquad\qquad $\mbox{\tt rake}[g] := \mbox{\tt rake}[g] \gg 1$
\item \qquad\qquad\qquad\qquad $S[g] := S[g] \,\,\&\,\, \mbox{\tt rake}[g]$
\item \qquad\qquad\qquad $S[g+1] := (S[g]\ll \tilde s[g]+1) \mid \left(1\,1\,1\gg c[g+1]-3\right)$
\item \qquad\qquad\qquad $G[g+1] := G[g] \mid \left(1\gg c[g+1]-2\right)$
\item \qquad\qquad\qquad $g := g+1$
\item \qquad\qquad\qquad $\mbox{\tt rake}[g] := G[g]$
\item \qquad\qquad\qquad $\tilde s_{last} := \tilde s[g] := 0$
\item \ {\bf return} \ $n$  
\end{enumerate}
}

\newpage

\noindent To produce the descendants from a given node of genus~$g$\, through this algorithm, 
the \emph{descending step} is not performed each time from scratch as it is given above. 
Indeed, the \emph{raking} process which must be made for a descendant need only be continued,
in the loop at lines~17--19, from the one already executed for the previous siblings, if any.
As a consequence, for each possible descendant, that is, for every non-negative value~\,$\tilde s[g]$ 
smaller than the multiplicity~\,$m[g]$,\, a limited number of operations are required: specifically, 
those at line 5 are carried out 
once for every such value, whereas those at lines 18 and 19 are carried out once for every value~\,$\tilde s[g]$ 
up to the last sibling, and those at the remaining lines inside the main loop are carried out at most once 
for each actual descendant.
Thus, the computation of \emph{all} descendants of the node requires $O(m[g])$ operations having
time complexity at most $O(\gamma)$. This is why, along with the fact that operations on binary strings
are very fast in practice, one should expect the algorithm to perform efficiently
when compared to other existing ones, for which either the construction of the data structure representing
a generic descendant takes time at least $O(\gamma\log\gamma)$, as in~\cite{FromentinHivert}, or otherwise 
the identification of its generators requires at best $O(\gamma)$ basic arithmetic operations, as in the 
setting given in \cite{Br:fibonacci,Br:bounds}.

\vspace{10truemm}

{\footnotesize
\noindent{\bf Recursive version \ $n_\gamma(G,\,S,\,c,\,g,\,m)$}\\[-8pt]
\begin{enumerate}\itemsep0.8truemm
\item \ {\bf if} \ \,$g=\gamma$\, \ {\bf then} \ \,$n:=1$
\item \ {\bf else} \ 
\item \qquad $n:=0$ \qquad $\tilde s_{last}:=0$ \qquad $\mbox{\tt rake} := G$
\item \qquad {\bf for} \ \,$\tilde s$\, \ {\bf from} \ \,$0$\, \ {\bf to} \ \,$m-1$\, \ {\bf do}
\item \qquad\qquad{\bf if} \ \,$S\,\,\&\,\,(1\gg\tilde s) \,\neq\, 0$\, \ {\bf then}
\item \qquad\qquad\qquad{\bf from} \ \,$\tilde s_{last}+1$\, \ {\bf to} \ \,$\tilde s$\, \ {\bf do}
\item \qquad\qquad\qquad\qquad $\mbox{\tt rake} := \mbox{\tt rake} \gg 1$
\item \qquad\qquad\qquad\qquad $S := S \,\,\&\,\, \mbox{\tt rake}$
\item \qquad\qquad\qquad $\tilde s_{last}:=\tilde s$
\item \qquad\qquad\qquad $\tilde c:=c+\tilde s+1$
\item \qquad\qquad\qquad {\bf if} \ \,$\tilde s =0$\, \ {\bf and} \ \,$m=c$\, \ {\bf then} \ \,$\tilde m:=\tilde c$\,
\ {\bf else} \ \,$\tilde m:=m$
\item \qquad\qquad\qquad $n:=n\,+\,n_\gamma\big(G\mid \left(1\gg\tilde c-2\right), \ 
\,(S\ll \tilde s+1) \mid \left(1\,1\,1\gg\tilde c-3\right), \  \,\tilde c, \ \,g+1, \ \,\tilde m\big)$
\item \ {\bf return} \ $n$
\end{enumerate}
}

\vspace{10truemm}


Let us conclude by comparing the running times of the two versions given above
for the seeds descending algorithm.
We also consider three known algorithms for computing the number of semigroups of
a given genus, which we briefly summarize as follows. We use the same notations
as in Sections~\ref{seeds} and \ref{results}.
\begin{itemize}

\item[$\bullet$] {\bf The \,\emph{Apéry set}\, algorithm.} 
A semigroup $\Lambda$ can be uniquely identified by its Apéry set 
\cite{Apery,RGS}, which is defined by taking, for each congruence class of integers modulo the multiplicity $\lambda_1$,
the smallest representative in~$\Lambda$. Then, this algorithm is based on the fact that
the generators of $\Lambda$ other than $\lambda_1$
necessarily lie in its Apéry set: they are the non-zero elements in the set which cannot be written, up to a multiple 
of~$\lambda_1$, as the sum of any two other non-zero elements in the set.
Thus, although the Apéry set of a descendant is immediately obtained 
from the Apéry set of $\Lambda$, the efficiency of this method is rather limited.

\item[$\bullet$] {\bf The \,\emph{generators tracking}\, algorithm.}
It is based on the construction given in \cite{Br:fibonacci,Br:bounds} 
for the semigroup tree
by keeping track of the set of generators at each descending step.
Specifically, the node attached to a semigroup $\Lambda$ can be encoded by its conductor \,$c$,
its multiplicity~$\lambda_1$ and a ternary array which is indexed 
by \,$1\leq i<c+\lambda_1$\, and whose value (say~\,$0$, $1$ or~$2$) at the $i$-th entry depends 
on whether \,$i$\, is a gap of $\Lambda$, a generator or a non-gap which is not a generator.
The array of a descendant $\tilde\Lambda=\Lambda\setminus\{\lambda_s\}$ inherits the values
from that of $\Lambda$ at every entry, except for the one with index~\,$i\!=\!\lambda_s$. 
In the very special case in which $\Lambda$ is ordinary and $s=1$, exactly
two new entries are required and both correspond to generators of~$\tilde\Lambda$.
Otherwise, \,$\lambda_s-c+1$\, new entries are required but only the last one
may correspond to a generator, according to Lemma 1 in~\cite{Br:bounds}.

\item[$\bullet$] {\bf The \,\emph{decomposition numbers}\, algorithm.} Fromentin and Hivert consider 
in \cite{FromentinHivert} the \emph{decomposition numbers}
$$d_\Lambda(x):=\#\left\{y\in \Lambda \,\mid\, x-y\in\Lambda, \ 2y\leq x\right\}$$
for \,$x\in\N$, attached to a semigroup $\Lambda$. These numbers, which are intimately related to the 
extensively used \,$\nu$-sequence in coding theory \cite{FR,KiPe,HoLiPe,Br:Acute,Br:ANote}, 
are exploited in a simple way in \cite{FromentinHivert} to efficiently explore the semigroup tree.
\end{itemize}

\vspace{1truemm}

\noindent For each algorithm, we made a first implementation using a 
\emph{depth first search} (DFS) exploration of the semigroup tree, as in \cite{FromentinHivert},
since this consumes much less memory than a \emph{breadth first search} exploration and so allows 
us to reach further in the computations.
We then implemented recursive versions for the algorithms.
All implementations were made in plain~\texttt{C}.
The following table displays the number of seconds that each of them required
to compute the number of semigroups of genus \,$g$\, for \,$30\leq g\leq 40$.

\vspace{6truemm}

\noindent\resizebox{\textwidth}{!}{\begin{tabular}{|l|r|r|r|r|r|r|r|r|r|r|r|}\hline
& {\bf 30} & {\bf 31} & {\bf 32} & {\bf 33} & {\bf 34} & {\bf 35} & {\bf 36} & {\bf 37} & 
{\bf 38} & {\bf 39} & {\bf 40} \\\hline
{\bf Ap\'ery} - DFS & 13 & 24 & 39 & 67 & 114 & 193 & 327 & 554 & 933 & 1577 & 2657 \\\hline
{\bf Ap\'ery} - recursive & 10 & 16 & 28 & 47 & 81 & 136 & 232 & 393 & 634 & 1071 & 1805 \\\hline
{\bf decomposition} - DFS & 10 & 16 & 27 & 46 & 79 & 131 & 222 & 373 & 626 & 1050 & 1762 \\\hline
{\bf generators} - DFS & 8 & 14 & 23 & 39 & 65 & 110 & 185 & 310 & 518 & 868 & 1448 \\\hline
{\bf decomposition} - recursive & 7 & 12 & 20 & 35 & 58 & 97 & 165 & 275 & 462 & 775 & 1297 \\\hline
{\bf generators} - recursive & 2 & 4 & 7 & 11 & 19 & 31 & 53 & 87 & 145 & 241 & 400 \\\hline
{\bf seeds} - DFS & 1 & 3 & 4 & 8 & 12 & 21 & 35 & 58 & 96 & 161 & 269 \\\hline
{\bf seeds} - recursive & 1 & 2 & 3 & 6 & 9 & 15 & 26 & 42 & 70 & 118 & 195 \\\hline
\end{tabular}}

\newpage



\vfill

\begin{footnotesize}
\begin{tabular}{l}
Maria Bras-Amor\'os\\
\texttt{maria.bras\,\footnotesize{$@$}\,urv.cat}\\[5pt]
Departament d'Enginyeria Inform\`atica i  Matem\`atiques  \\
Universitat Rovira i Virgili\\
Avinguda dels Països Catalans, 26\\
E-43007 Tarragona\\[30pt]
Julio Fern\'andez-Gonz\'alez\\
\texttt{julio.fernandez.g\,\footnotesize{$@$}\,upc.edu}\\[5pt]
Departament de Matem\`atiques  \\
Universitat Polit\`ecnica de Catalunya  \\
EPSEVG -- Avinguda V\'ictor Balaguer, 1\\
E-08800 Vilanova i la Geltr\'u\\[20pt]
\end{tabular}
\end{footnotesize}

\end{document}